\documentclass[12pt]{article}

\usepackage{amssymb}
\usepackage{amsthm}
\usepackage{amsmath}
\usepackage{graphicx}
\usepackage{fullpage}
\usepackage{color}
\usepackage{enumerate}
 \numberwithin{equation}{section}
\usepackage{booktabs}
\allowdisplaybreaks

\usepackage[pdftex]{hyperref}
 \usepackage{mathtools}
 \mathtoolsset{showonlyrefs}
 \usepackage[nocompress]{cite}

\theoremstyle{plain}
\newtheorem{thm}{Theorem}[section]
\newtheorem{cor}[thm]{Corollary}

\newtheorem{lem}[thm]{Lemma}
\newtheorem{prop}[thm]{Proposition}

\newtheorem*{theorem}{Theorem}

\theoremstyle{definition}

\theoremstyle{remark}

\newtheorem{rem}[thm]{Remark}

\newcommand{\N}{\mathbb{N}}
\newcommand{\R}{\mathbb{R}}

\newcommand{\bp}{\begin{proof}[\ensuremath{\mathbf{Proof}}]}
\newcommand{\bs}{\begin{proof}[\ensuremath{\mathbf{Solution}}]}
\newcommand{\ep}{\end{proof}}

\newcommand{\be}{\begin{equation}}
\newcommand{\ee}{\end{equation}}

\DeclareMathOperator*{\argmin}{arg\,min}

\begin{document}

\title{Continuous time approximation of Nash equilibria}

\author{Romeo Awi\footnote{Department of Mathematics, Hampton University.}\;, Ryan Hynd\footnote{Department of Mathematics, University of Pennsylvania. }\;, and Henok Mawi\footnote{Department of Mathematics, Howard University.}\;}

\maketitle

\begin{abstract} 
We consider the problem of approximating Nash equilibria of $N$ functions $f_1,\dots, f_N$ of $N$ variables.  In particular, we deduce conditions under which systems of the form
$$
\dot u_j(t)=-\nabla_{x_j}f_j(u(t))
$$
$(j=1,\dots, N)$  are well posed and in which the large time limits of their solutions $u(t)=(u_1(t),\dots, u_N(t))$ are Nash equilibria for $f_1,\dots, f_N$.  To this end, we will invoke the theory of maximal monotone operators. We will also identify an application of these ideas in game theory and show how to approximate equilibria in some game theoretic problems in function spaces. 
\end{abstract}

\section{Introduction}
Let us first recall the notion of a Nash equilibrium. Consider a collection of $N$ sets $X_1,\dots, X_N$ and define
$$
X:=X_1\times \dots \times X_N.
$$
For a given $x=(x_1,\dots, x_N)\in X$, $j\in \{1,\dots, N\}$, and $y_j\in X_j$, we will use the notation $(y_j,x_{-j})$ for the point in $X$ in which 
$y_j$ replaces $x_j$ in the coordinates of $x$. That is,
$$
(y_j,x_{-j}):=(x_1,\dots, x_{j-1},y_j ,x_{j+1},\dots, x_N).
$$
A collection of functions $f_1,\dots, f_N: X\rightarrow \R$ has a {\it Nash equilibrium} at $x\in X$ provided 
$$
f_j(x)\le f_j(y_j,x_{-j})
$$
for all $y_j\in X_j$ and $j=1,\dots, N$.

\par Nash recognized that an equilibrium is a fixed point of the set valued mapping
$$
X\ni x\mapsto \argmin_{y\in X}\left\{\sum^N_{j=1}f_j(y_j,x_{-j})\right\}.
$$
In particular, he applied Kakutani's fixed point theorem \cite{MR4776} to show that if $X_1,\dots, X_N$ are nonempty, convex, compact subsets of Euclidean space and each $f_1,\dots, f_N$ is multilinear, then $f_1,\dots, f_N$ has a Nash equilibrium at some $x\in X$ \cite{MR43432,MR31701}.   Nash was interested in multilinear $f_1,\dots, f_N$ as they correspond to the expected cost of players assuming mixed strategies in noncooperative games. More generally, his existence theorem holds if each $f_j: X\rightarrow \R$ is continuous and
\be\label{fjayConvex}
X\ni y\mapsto \sum^N_{j=1}f_j(y_j,x_{-j}) \text{ is convex for each $x\in X$. }
\ee

\par As the existence of a Nash equilibrium is due to a nonconstructive fixed point theorem, it seems unlikely that there are good ways to approximate these points.  Indeed, it has been established that the approximation of Nash equilibria is computationally challenging \cite{MR2506524,MR2510264,10.1145/779928.779933}.   Nevertheless, we contend that there is a nontrivial class of functions $f_1,\dots, f_N$ for which the approximation of Nash equilibria is at least theoretically feasible.  We will explain below that a sufficient condition for the approximation of Nash equilibria is that $f_1,\dots, f_N$ satisfy
\be\label{f1thrufNmonotone}
\sum^N_{j=1}(\nabla_{x_j}f_j(x)-\nabla_{x_j}f_j(y))\cdot (x_j-y_j)\ge 0
\ee
for each $x,y\in X$.  Here we are considering each $X_j$ as a subset of a Euclidean space and use `$\cdot$' to denote the dot products on any of these spaces. 

\par Our approach starts with the observation that if \eqref{fjayConvex} holds and $X_1,\dots, X_N$ are convex, then $x$ is a Nash equilibrium if and only if 
\be\label{f_1thrufNVI}
\sum^N_{j=1}\nabla_{x_j}f_j(x)\cdot (y_j-x_j)\ge 0
\ee
for each $y\in X$. As a result, it is natural to consider the differential inequalities 
\be\label{f1thrufNODE}
\sum^N_{j=1}(\dot u_j(t)+\nabla_{x_j}f_j(u(t)))\cdot (y_j-u_j(t))\ge 0
\ee
for $t\ge 0$ and $y\in X$.  Here 
$$
u:[0,\infty)\rightarrow X;t\mapsto (u_1(t),\dots, u_N(t))
$$
is the unknown.  We will below argue that for appropriately chosen initial condition $u(0)\in X$, the {\it Ces\`aro mean} of $u$
\be
t\mapsto\frac{1}{t}\int^t_0 u(s)ds
\ee
will converge to a Nash equilibrium of $f_1,\dots, f_N$ as $t\rightarrow\infty$ provided that \eqref{f1thrufNmonotone} holds. 

\subsection{A few concrete examples}
\par  An elementary example which illustrates this approach may be observed when $X_1=X_2=[-1,1]$, $X=[-1,1]^2$, 
$$
f_1(x_1,x_2)=x_1x_2,\;\;\text{and}\;\;f_2(x_1,x_2)=-x_1x_2.
$$ 
It is not hard to check that the origin $(0,0)$ is the unique Nash equilibrium of $f_1,f_2$ and that $f_1,f_2$ satisfy \eqref{f1thrufNmonotone} (with equality holding) for each $x,y\in[-1,1]^2$.  We now seek an absolutely continuous
path 
$$
u:[0,\infty)\rightarrow [-1,1]^2
$$ 
such that 
\be\label{ExampleODE}
(\dot u_1(t)+u_2(t)) (y_1-u_1(t))+(\dot u_2(t)-u_1(t)) (y_2-u_2(t))\ge 0
\ee
holds for almost every $t\ge 0$ and each $(y_1,y_2)\in [-1,1]$. 

\par It turns out that if 
$$
u_1(0)^2+u_2(0)^2\le 1,
$$
then the unique solution of these differential inequalities is 
$$
\begin{cases}
\displaystyle u_1(t)=u_1(0)\cos(t)-u_2(0)\sin(t)
\\\\
\displaystyle u_2(t)=u_2(0)\cos(t)+u_1(0)\sin(t).
\end{cases}
$$
This solution simply parametrizes the circle of radius $\sqrt{u_1(0)^2+u_2(0)^2}$. In particular,
$$
\lim_{t\rightarrow \infty}\frac{1}{t}\int^t_0u(s)ds=(0,0).
$$
For general initial conditions $(u_1(0),u_2(0))\in [-1,1]^2$, it can be shown that $u_1(s)^2+u_2(s)^2\le 1$ in a finite time $s\ge 0$.  Refer to Figure \ref{Fig1} for a schematic. It then follows that the same limit above holds for this solution.
\begin{figure}[h]
\centering
 \includegraphics[width=.6\textwidth]{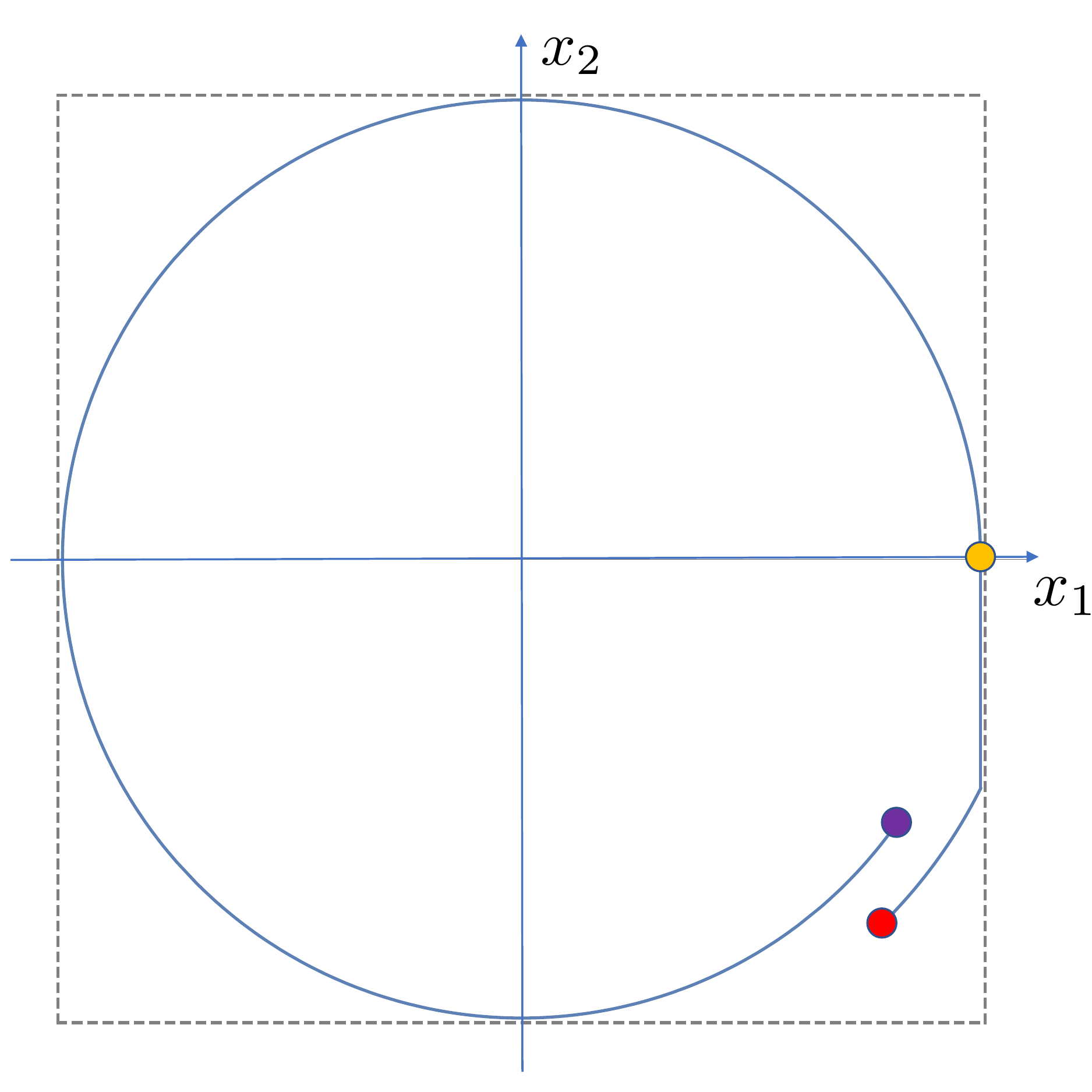}
 \caption{Plot of a solution $u$ of \eqref{ExampleODE} where the initial position of $u$ is indicated in red.  This solution starts out on a circle centered at the origin traversing it counterclockwise until it hits the boundary line segment $(x_1=1, -1\le x_2\le 1)$. Then it proceeds upwards along this boundary line segment until it arrives at some time $s$ at the orange marker which is located at $(1,0)$.  For times $t\ge s$, the position $u(t)$ as shown in purple remains on the unit circle traversing it counterclockwise. }\label{Fig1}
\end{figure}
\\
\par The method we present can be extended in certain cases when $X_1,\dots, X_N$ are not compact. For example, suppose $X_1=X_2=\R$, 
$$
f_1(x_1,x_2)=\frac{1}{2}x_1^2-x_1x_2-x_1,\;\;\text{and}\;\; f_2(x_1,x_2)=\frac{1}{2}x_2^2+x_1x_2-2x_2.
$$
It is easy to check that $f_1,f_2$ satisfy \eqref{f1thrufNmonotone}.  Furthermore, 
$$
\begin{cases}
\partial_{x_1}f_1(x_1,x_2)=x_1-x_2-1=0\\\\
\partial_{x_2}f_2(x_1,x_2)=x_2+x_1-2=0
\end{cases}
$$
provided
$$
x_1=\frac{3}{2}\quad \text{and}\quad x_2=\frac{1}{2}.
$$
As a result, this is the unique Nash equilibrium of $f_1, f_2$. 

\par  We  note that the solution of
$$
\begin{cases}
\dot u_1(t)=-(u_1(t)-u_2(t)-1)\\\\
\dot u_2(t)=-(u_2(t)+u_1(t)-2)
\end{cases}
$$
is given by
$$
\begin{cases}
\displaystyle u_1(t)=\frac{3}{2}+\left(u_1(0)-\frac{3}{2}\right)e^{-t}\cos(t)+
\left(u_2(0)-\frac{1}{2}\right)e^{-t}\sin(t)\\\\
\displaystyle u_2(t)=\frac{1}{2}+\left(u_2(0)-\frac{1}{2}\right)e^{-t}\cos(t)+
\left(\frac{3}{2}-u_1(0)\right)e^{-t}\sin(t).
\end{cases}
$$
It is now clear that 
$$
\lim_{t\rightarrow\infty}(u_1(t),u_2(t))=\left(\frac{3}{2},\frac{1}{2}\right).
$$
In particular, we do not need to employ the Ces\`aro mean of $u$ in order to approximate the Nash equilibrium of $f_1,f_2$.

\subsection{A general setting}\label{GenSettingSubsection} 
It just so happens that our method does not rely on the spaces $X_1,\dots, X_N$ being finite dimensional. Consequently, we will consider the following version of our approximation problem. Let $V$ be a reflexive Banach space over $\R$ with norm $\|\cdot \|$ and continuous dual $V^*$.  We further suppose 
\be
X_1,\dots, X_N\subset V\; \text{are closed and convex with nonempty interiors}
\ee
and consider $N$ functions  
\be
f_1,\dots, f_N: X\rightarrow \R
\ee
which satisfy
\be\label{StandardAssumpf1thrufN}
\begin{cases}
f_1,\dots, f_N\;\text{are weakly lowersemicontinuous},
\\\\
\displaystyle X\ni y\mapsto \sum^N_{j=1}f_j(y_j,x_{-j})\;\text{is convex for each $\displaystyle x\in X$,}
\\\\
X\ni x\mapsto \displaystyle \sum^N_{j=1} f_j(y_j,x_{-j})\;\text{is weakly continuous for each $y\in X$, and}
\\\\
\displaystyle\sum^N_{j=1}(\partial_{x_j}f_j(x)-\partial_{x_j}f_j(y), x_j-y_j)\ge 0\;\; \text{for each $x,y\in X$}.
\end{cases}
\ee
In the last condition listed above, we mean 
$$
\sum^N_{j=1}(p_j-q_j, x_j-y_j)\ge 0
$$
for each 
$$
\displaystyle p_j\in \partial_{x_j}f_j(x):=\left\{\zeta\in V^*: f_j(z,x_{-j})\ge f_j(x)+(\zeta,z-x_j),\; z\in X_j\right\}
$$ 
and $q_j\in \partial_{x_j}f_j(y)$ for $j=1,\dots, N$.

\par When $V$ is finite dimensional, we naturally identify $V$ and $V^*$ with Euclidean space of the same dimension. Alternatively, when $V$ is infinite dimensional, we will suppose there is a real Hilbert space $H$ for which 
$$
V\subset H\subset V^*
$$
and $V\subset H$ is continuously embedded.  It is with respect to this space $H$ in which we consider the following initial value problem: for a given $u^0\in X$, find an absolutely continuous $u:[0,\infty)\rightarrow H^N$ such that 
\be\label{flowIVP}
\begin{cases}
u(0)=u^0,\;\text{and}
\\\\
 \displaystyle\sum^N_{j=1}(\dot u_j(t)+ \partial_{x_j}f_j(u(t)), y_j-u_j(t))\ge 0\; \text{for a.e. $t\ge 0$ and each $y\in X$.}
\end{cases}
\ee

\par Our central result involving this initial value problem is as follows.  In this statement, we will make use of the set
\be\label{OurDomain}
{\cal D}=\left\{x\in X:\bigcap^N_{j=1}\left(\partial_{x_j}f_j(x)+n_{X_j}(x_j)\right)\cap H\neq \emptyset \right\}.
\ee
Here $n_{X_j}: V\rightarrow 2^{V^*}$ is the normal cone of $X_j$ defined as
\be\label{NormCone}
n_{X_j}(z):=\{\zeta\in V^*: ( \zeta,y-z)\le 0\;\text{all $y\in X_j$}\}
\ee
for $z\in V$ $(j=1,\dots, N)$.

\begin{thm}\label{mainthm1}
Assume $f_1,\dots, f_N$ satisfy \eqref{StandardAssumpf1thrufN}.
\begin{enumerate}[$(i)$] 

\item For any $u^0\in {\cal D}$ there is a unique Lipschitz continuous $u:[0,\infty)\rightarrow H^N$ satisfying $u(t)\in {\cal D}$ 
for each $t\ge 0$ and the initial value problem \eqref{flowIVP}. \\

\item If $f_1,\dots, f_N$ has a Nash equilibrium, then 
\be
\frac{1}{t}\int^t_0 u(s)ds
\ee
converges weakly in $H$ to a Nash equilibrium of $f_1,\dots, f_N$ as $t\rightarrow \infty$.
\end{enumerate}
\end{thm}
\par We note that if $X_1=\cdots=X_N=V$, then $n_{X_1}\equiv\cdots \equiv n_{X_N} \equiv \{0\}$. In this case 
\be
{\cal D}=\left\{x\in X:\bigcap^N_{j=1}\left(\partial_{x_j}f_j(x)\cap H\right)\neq \emptyset \right\},
\ee
and the initial value problem reduces to 
\be\label{flowIVP2}
\begin{cases}
u(0)=u^0,\;\text{and}
\\\\
 \dot u_j(t)+ \partial_{x_j}f_j(u(t))\ni 0\; \text{for a.e. $t\ge 0$ and each $j=1,\dots, N$}.
\end{cases}
\ee
We will also explain that if in addition $(\partial_{x_1}f_1,\dots, \partial_{x_N}f_N)$ is
single-valued, everywhere-defined, monotone, and hemicontinuous we can obtain the same result as above without assuming \eqref{StandardAssumpf1thrufN}. This follows directly from a theorem of Browder and Minty.

\par We will phrase our initial value problem \eqref{flowIVP} as the evolution generated by a monotone operator on $H^N$ with domain ${\cal D}$. According to pioneering work of Br\'{e}zis \cite{MR0348562}, the crucial task will be  to verify the maximality of this operator.  Regarding the large time behavior of solutions, the convergence to equilibria of maximal monotone operators by the Ces\`aro means of semigroups they generate originates in the work of Br\'{e}zis and Baillon \cite{MR394328}.  We also refer the reader to Chapter 3 of \cite{MR755330} which gives a detailed discussion of this phenomenon.  

\par Of course, the asymptotic statements we made above are predicated on the existence of a Nash equilibrium. This is only guaranteed to be the case if $X$ is weakly compact or if there is some $\theta>0$ such that 
\be
\displaystyle\sum^N_{j=1}(\partial_{x_j}f_j(x)-\partial_{x_j}f_j(y), x_j-y_j)\ge \theta\|x-y\|^2
\ee
for each $x,y\in X$. When such coercivity holds, $f_1,\dots, f_N$ has a unique Nash equilibrium and solutions of the initial value problem \eqref{flowIVP} converge exponentially fast in $H^N$ to this 
equilibrium point. 

\par We will prove Theorem \ref{mainthm1} in the following section. Then we will apply this result to flows in Euclidean spaces of the form \eqref{f1thrufNODE} in section \ref{FiniteDimSect}. Finally, we will consider examples in Lebesgue and Sobolev spaces in section \ref{InFiniteDimSect}. A prototypical collection of functionals we will study is
\be
f_j(v)=\int_{\Omega}\frac{1}{2}|\nabla v_j|^2+F_j(v)dx
\ee
for $j=1,\dots, N$, where $v=(v_1,\dots, v_N)\in H_0^1(\Omega)^N$ and $\Omega\subset \R^d$ is a bounded domain. For this particular example, we will argue that if $F_1,\dots, F_N:\R^N\rightarrow\R$ satisfy suitable growth and monotonicity conditions, then $f_1,\dots,f_N$ has a unique Nash equilibrium which can be approximated by solutions $u_1,\dots, u_N:\Omega\times[0,\infty)\rightarrow\R$ of the parabolic initial/boundary value problem
\be
\begin{cases}
\partial_t u_j =\Delta u_j-\partial_{z_j}F_j(u) \;\; &\text{in}\;\Omega\times(0,\infty)\\
\;\;\;u_j=0 \;\; & \text{on}\; \partial\Omega\times(0,\infty)\\
\;\;\;u_j=u^0_j \;\; & \text{on}\; \Omega\times\{0\}
\end{cases}
\ee
for appropriately chosen initial conditions $u^0_1,\dots, u^0_N: \Omega\rightarrow\R$.

\par We also remark that in finite dimensions, similar results were known to Fl\aa m \cite{MR1201625}. In particular, Fl\aa m seems to be the first author to identify the important role of the monotonicity condition \eqref{f1thrufNmonotone} in approximating Nash equilibrium.  However, Rosen appears to be the first to consider using equations such as \eqref{f1thrufNODE} to approximate equilibrium points \cite{MR194210}. In addition, we would like to point out that the finite dimensional variational inequality \eqref{f_1thrufNVI} and its relation to Nash equilibrium is studied in depth in the monograph by Facchinei and Pang \cite{MR1955648}. In infinite dimensions, there have also been several recent works which discuss discrete time approximations for variational inequalities that correspond to Nash equilibrium \cite{MR3108427, MR2956124,MR3285906,MR2486094,MR3982683}.

\par {\bf Acknowledgements}: This material is based upon work supported by the National Science Foundation
under Grants DMS-1440140, DMS-1554130, and HRD-1700236, the National Security Agency under Grant No.
H98230-20-1-0015, and the Sloan Foundation under Grant No. G-2020-12602 while the
authors participated in a program hosted by the Mathematical Sciences Research Institute in Berkeley, California, during the summer of 2020.

\section{Approximation theorem}
In this section, we will briefly recall the notion of a maximal monotone operator on a Hilbert space and state a few key results 
for these operators.  Then we will apply these results to prove Theorem \ref{mainthm1} and a few related corollaries. 

\subsection{Maximal monotone operators on a Hilbert space}
Let $H$ be Hilbert space with inner product $\langle\cdot,\cdot\rangle$ and norm $|\cdot |$. We will denote $2^H$ as the power set or collection of all subsets of $H$.  We recall that $B: H\rightarrow 2^H$ is monotone if 
$$
\langle p-q,x-y\rangle\ge 0
$$
for all $x,y\in H$, $p\in Bx$, and $q\in By$. Moreover, $B$ is {\it maximally monotone} if the only monotone
$C: H\rightarrow 2^H$ such that $Bx\subset Cx$ for all $x\in H$ is $C=B$.   Minty's well known maximality criterion  is as follows \cite{MR169064}.

\begin{theorem} [Minty's Lemma]
A monotone operator $B:H\rightarrow 2^H$ is maximal if and only if for each $y\in H$, there is a unique $x\in H$ such that 
$$
x+Bx\ni y.
$$
\end{theorem}

\par Let us now recall a fundamental theorem for the initial value problem 
\be\label{abstractIVP}
\begin{cases}
\dot u(t)+Bu(t)\ni 0, \; \text{a.e.}\; t\ge 0\\
u(0)=u^0.
\end{cases}
\ee
As discussed in Chapter II of \cite{MR0348562} and Chapter 3 of \cite{MR755330}, the initial value problem is well-posed provided that $B$ is maximally monotone and $u^0$ is an element of the domain of $B$
$$
D(B)=\{x\in H: Bx\neq \emptyset\}.
$$
\begin{theorem} [Br\'{e}zis's Theorem]
Suppose $B:H\rightarrow 2^H$ is maximally monotone and $u^0\in D(B)$. Then there is 
 a unique Lipschitz continuous  
$$
u:[0,\infty)\rightarrow H
$$
solution of \eqref{abstractIVP} such that $u(t)\in D(B)$ for all $t\ge 0$. Moreover,
\be\label{DerivativeBoundYou}
|\dot u(t)|\le \min\{|p|: p\in Bu^0\}
\ee
for almost every $t\ge 0$. 
\end{theorem}
\begin{rem}
The full statement of Br\'{e}zis's Theorem (Theorem 3.1 of \cite{MR755330}) is more extensive than what is written above. 
\end{rem}
\begin{rem}
We note that the right hand side of \eqref{DerivativeBoundYou} is finite as the images of maximal monotone operators are closed and convex subsets of $H$. 
\end{rem}

\par We also note that the operator $B$ generates a contraction. Let $u$ be a path as described in Br\'ezis' Theorem, and suppose $v: [0,\infty)\rightarrow H$ is any other Lipschitz continuous path with $v(t)\in D(B)$ for $t\ge 0$ and
$$
\dot v(t)+Bv(t)\ni 0, \; \text{a.e.}\; t\ge 0.
$$
Then 
\be\label{contractionProp}
|u(t)-v(t)|\le |u(0)-v(0)|, \quad t\ge 0.
\ee
If, in addition, there is some $\lambda>0$ such that 
\be\label{Bcoercive}
\langle p-q,x-y\rangle\ge \lambda |x-y|^2
\ee
for all $x,y\in H$, $p\in Bx$, and $q\in By$, then \eqref{contractionProp} can be improved to 
\be\label{ExpContraction}
|u(t)-v(t)|\le e^{-\lambda t} |u(0)-v(0)|, \quad t\ge 0.
\ee

\par Remarkably, it is also possible to use solutions of the initial value problem (as described in Br\'ezis's Theorem) to approximate equilibria of $B$. These are points $y\in H$ such that  
$$
By\ni 0.
$$
The following theorem was proved in \cite{MR394328}.
\begin{theorem}[The Baillon-Br\'{e}zis Theorem]
Suppose $B$ has an equilibrium point and $u$ is a solution of the initial value problem \eqref{abstractIVP}.  Then
$$
\frac{1}{t}\int^t_0u(s)ds
$$
converges weakly in $H$ to some equilibrium point of $B$.  
\end{theorem}
\begin{rem}\label{ExponentialConvRemark}
If $B$ satisfies \eqref{Bcoercive}, it is plain to see that $By\ni 0$ can have at most one solution. In this case, 
$$
|u(t)-y|\le e^{-\lambda t} |u(0)-y|, \quad t\ge 0
$$
for any solution of the initial value problem \eqref{abstractIVP}. This follows from \eqref{ExpContraction} as $v(t)=y$ is a solution of the initial value 
problem with $v(0)=y$. 
\end{rem}

\subsection{Proof of Theorem \ref{mainthm1}}
Let us now suppose the spaces $V$, $X_1,\dots, X_N$, and $H$ are as described in the subsection \ref{GenSettingSubsection} of the introduction.  We will further assume $f_1,\dots, f_N: X\rightarrow \R$ are given functions which satisfy \eqref{StandardAssumpf1thrufN}. An elementary but important observation we will use is that a given $y\in H$ induces a linear form in $V^*$ 
by the formula 
$$
(y,x):=\langle y,x\rangle, \quad (x\in V).
$$
That this linear form is continuous is due to the continuity of the embedding $V\subset H$. For convenience, we will use $\|\cdot \|$ to denote both the norm on $V$ and the associated norm on $V^N$. Namely, we'll write
$$
\|x\|=\left(\sum^N_{j=1}\|x_j\|^2\right)^{1/2}
$$
for $x=(x_1,\dots, x_N)\in V^N$. We will also use this convention for the inner product $\langle\cdot,\cdot\rangle$ and norm $|\cdot |$ on $H$ and on $H^N$ and the induced norm $\|\cdot \|_*$ on $V^*$ and on $(V^*)^N$.

\par Let us specify $A: H^N\rightarrow 2^{H^N}$ via
\be\label{ourOperator}
Ax=
\begin{cases}
\left(\partial_{x_1}f_1(x)+n_{X_1}(x_1)\right)\cap H\times \dots\times\left(\partial_{x_N}f_N(x)+n_{X_N}(x_N)\right)\cap H,& x\in{\cal D}\\
\emptyset, & x\not\in{\cal D}
\end{cases}
\ee
for $x\in H^N$. Here ${\cal D}$ is defined in \eqref{OurDomain}.
\begin{lem}
The domain of $A$ is ${\cal D}$, and $A$ is monotone. Moreover, an absolutely continuous
$u:[0,\infty)\rightarrow H^N$ solves
\be\label{OurabstractIVP}
\begin{cases}
\dot u(t)+Au(t)\ni 0, \; \text{a.e.}\; t\ge 0\\
u(0)=u^0
\end{cases}
\ee
if and only it solves \eqref{flowIVP}.
\end{lem}
\begin{proof}
By definition, the domain of $A$ is ${\cal D}$.  Suppose $x,y\in {\cal D}$, 
$$
\zeta_j\in\left(\partial_{x_j}f_j(x)+n_{X_j}(x_j)\right)\cap H,\; \text{and}\; \xi_j\in\left(\partial_{x_j}f_j(y)+n_{X_j}(y_j)\right)\cap H
$$
for $j=1,\dots, N$.  In view of the last condition listed in \eqref{StandardAssumpf1thrufN} and the fact that each $n_{X_j}$ is monotone,
\begin{align}
\sum^N_{j=1}\langle \zeta_j-\xi_j,x_j-y_j\rangle&=
\sum^N_{j=1}( \zeta_j-\xi_j,x_j-y_j)\ge 0.
\end{align}
Thus, $A$ is monotone. 
\par Note that if $u$ solves \eqref{OurabstractIVP},  then
$$
 -\dot u_j(t)\in ( \partial_{x_j}f_j(u(t))+n_{X_j}(u_j(t)))\cap H\subset \partial_{x_j}f_j(u(t))+n_{X_j}(u_j(t))
 $$
 for almost every $t\ge 0$ and each $j=1,\dots, N$. It now follows from the definition of the normal cone that $u$ is a solution of \eqref{flowIVP}.  Conversely, if $u$ solves \eqref{flowIVP} then 
 $$
  -\dot u_j(t)\in \partial_{x_j}f_j(u(t))+n_{X_j}(u_j(t))
 $$ 
for almost every $t\ge 0$ and each $j=1,\dots, N$. As $u$ is absolutely continuous, $-\dot u(t)\in H$ for almost every $t\ge 0$, 
and therefore,  $-\dot u(t)\in Au(t)$ for almost every $t\ge 0$.
\end{proof}
In view of Br\'ezis' Theorem and the Baillon-Br\'ezis Theorem, we can conclude Theorem \ref{mainthm1} once we verify that $A$ is maximal.  To this end, we will verify the hypotheses of Minty's Lemma.  
\begin{proof}[Proof of Theorem \ref{mainthm1}] 
For a given $y\in H^N$, it suffices to show there is $x\in X$ such that 
\be\label{SolvingxplusAxequaly}
\displaystyle\sum^N_{j=1}(x_j+\partial_{x_j}f_j(x)- y_j, z_j-x_j)\ge 0
\ee
for each $z\in X$.  In this case, $y_j-x_j\in H$ and
$$
y_j-x_j\in \partial_{x_j}f_j(x)+n_{X_j}(x_j)
$$
for $j=1,\dots, N$ so that $x\in {\cal D}$ and $y-x\in Ax$. 

\par In order to solve \eqref{SolvingxplusAxequaly}, we will employ the auxiliary functions defined by
$$
g_j(x):= \displaystyle\frac{1}{2}\|x_j\|^2-(y_j,x_j)+f_j(x)
$$
for $x\in X$ and $j=1,\dots, N$. In view of \eqref{StandardAssumpf1thrufN},
\be\label{StandardAssumph1thruhN}
\begin{cases}
g_1,\dots, g_N\;\text{are weakly lowersemicontinuous},
\\\\
X\ni y\mapsto \displaystyle\sum^N_{j=1}g_j(y_j,x_{-j})\;\text{is convex for each $\displaystyle x\in X$,}
\\\\
X\ni x\mapsto \displaystyle\sum^N_{j=1}g_j(y_j,x_{-j})\;\text{is weakly continuous for each $y\in X$, and}
\\\\
\displaystyle\sum^N_{j=1}(\partial_{x_j}g_j(x)-\partial_{x_j}g_j(y), x_j-y_j)\ge \|x-y\|^2\;\; \text{for each $x,y\in X$}.
\end{cases}
\ee
Note that $x\in X$ is a Nash equilibrium of $g_1,\dots, g_N$ if and only if $x$ solves \eqref{SolvingxplusAxequaly}. Since we are assuming that each $X_j$ has nonempty interior, this follows from the Pshenichnii--Rockafellar conditions for the minimum of a proper, lowersemicontinuous, convex function on a closed subset of a Banach space (Theorem 4.3.6 of \cite{MR2144010}). Consequently, we now aim to show that $g_1,\dots, g_N$ has a Nash equilibrium. 

\par  For each $r>0$, set 
$$
X^r:=\{x\in X: \|x\|\le r\}.
$$
It is clear that $X^r$ is convex and weakly closed.  And as each $X_1,\dots, X_N$ is nonempty, $X^r$ is nonempty for all $r$ greater or equal to some fixed $s>0$.  Let us consider the map 
$\Phi^r: X^r\rightarrow 2^{X^r}$ specified as
$$
\Phi^r(x):= \argmin_{y\in X^r}\left\{\sum^N_{j=1}g_j(y_j,x_{-j})\right\}
$$
$(x\in X^r)$ for $r>s$.

\par The first and second properties of $g_1,\dots, g_N$ listed in \eqref{StandardAssumph1thruhN} imply that $\Phi^r(x)\neq \emptyset$ and convex for any $x\in X^r$ and $r>s$; and the first and third properties imply that the graph of $\Phi^r$ is weakly closed. Indeed let us suppose $(x^k)_{k\in \N}, (y^k)_{k\in \N}\subset X$, $x^k\rightharpoonup x$, $y^k\rightharpoonup y$, and
$$
\sum^N_{j=1}g_j(y^k_j,x^k_{-j}) \le \sum^N_{j=1}g_j(z_j,x^k_{-j}) 
$$
for each $z\in X$ with $\|z\|\le r$ and all $k\in \N$. We can then send $k\rightarrow\infty$ to get 
$$
\sum^N_{j=1}g_j(y_j,x_{-j}) \le \sum^N_{j=1} g_j(z_j,x_{-j}).
$$

\par The Kakutani-Glicksberg-Fan theorem \cite{MR4776,MR46638,MR47317} then implies there is $x^r\in \Phi^r(x^r)$. In particular, 
\be\label{NashEqCondhjay}
\sum^N_{j=1}g_j(x^r)\le \sum^N_{j=1}g_j(z_j,x^r_{-j})
\ee
for each $z\in X$ with $\|z\|\le r$.  It follows that there are $p^r_j\in \partial_{x_j}g_j(x^r)$ for each $j=1,\dots, N$ such that 
$$
\sum^N_{j=1}(p^r_j, z_j-x^r_j)\ge 0
$$
for all $z\in X$ with $\|z\|\le r$. Since  $u^0\in {\cal D}$, there are
$$
w_j\in \partial_{x_j}g_j(u^0)
$$
for $j=1,\dots, N$.

\par Note that the fourth listed property of $g_1,\dots, g_N$ in \eqref{StandardAssumph1thruhN} gives 
\begin{align*}
\|x^r-u^0\|^2&\le \sum^N_{j=1}(p^r_j-w_j,x^r_j-u^0_j)\\
&\le \sum^N_{j=1}(-w_j,x^r_j-u^0_j)\\
&\le \|w\|_{*}\|x^r-u^0\|
\end{align*}
for all $r$ sufficiently large.  As a result, there is a sequence $(r_k)_{k\in\N}$ which increases to infinity such that $(x^{r_k})_{k\in \N}\subset X$ converges weakly to some $x\in X$. Upon sending $r=r_k\rightarrow\infty$ in \eqref{NashEqCondhjay}, we find that $x$ is the desired Nash equilibrium of $g_1,\dots, g_N$. 
\end{proof}
\begin{cor}
If $X$ is weakly compact, then $f_1,\dots, f_N$ has a Nash equilibrium.  
\end{cor}
\begin{proof}
If $X$ is weakly compact, we can repeat the argument given in the proof 
above used to show $g_1,\dots, g_N: X^r\rightarrow \R$ has a Nash equilibrium.  All we would need to do is to replace 
 $X^r$ with $X$  and $g_1,\dots, g_N$ with $f_1,\dots, f_N$. 
\end{proof}
\begin{cor}\label{ExistenceNashCoercive}
Suppose there is $\theta>0$ such that
\be\label{f1throughfNcoercive}
\displaystyle\sum^N_{j=1}(\partial_{x_j}f_j(x)-\partial_{x_j}f_j(y), x_j-y_j)\ge \theta\|x-y\|^2
\ee
for each $x,y\in X$. Then $f_1,\dots, f_N$ has a unique Nash equilibrium $z\in X$.  Moreover,  if $u: [0,\infty)\rightarrow H^N$ is a solution of the initial value problem as described in Theorem \ref{mainthm1}, there is $\lambda>0$ such that 
\be\label{ExpConvtoNash}
|u(t)-z|\le e^{-\lambda t}|u(0)-z|
\ee
for each $t\ge 0$.
\end{cor}
\begin{proof}
We can repeat the argument given in the proof of Theorem \ref{mainthm1} that shows $g_1,\dots, g_N$ has a Nash equilibrium to conclude that $f_1,\dots, f_N$ has a Nash equilibrium. 

\par If $x$ is a Nash equilibrium of $f_1,\dots, f_N$, then
$$
\sum^N_{j=1}(\partial_{x_j}f_j(x), z_j-x_j)\ge0
$$
for each $z\in X$. Likewise, 
$$
\sum^N_{j=1}(\partial_{x_j}f_j(y), w_j-y_j)\ge0
$$
for each $w\in X$ if $y$ is another Nash equilibrium. Choosing $z=y$ and $w=x$ and adding these inequalities give
$$
 \theta\|x-y\|^2\le \sum^N_{j=1}(\partial_{x_j}f_j(x)-\partial_{x_j}f_j(y), x_j-y_j)\le0.
$$

\par As $V^N\subset H^N$ is continuously embedded, there is a constant $C>0$ such that 
$$
|x|\le C\|x\|, \quad x\in V^N.
$$
Setting
$$
\lambda:=\frac{\theta}{C^2},
$$
gives
$$
\langle Ax-Ay,x-y\rangle\ge \theta\|x-y\|^2\ge \lambda|x-y|^2 
$$
for $x,y\in D(A)$. The inequality \eqref{ExpConvtoNash} now follows from Remark \ref{ExponentialConvRemark}.
\end{proof}
We now consider the special case mentioned in the introduction.  In the statement below, we will suppose that $V$ and $H$ are as above.  However, we will not assume $f_1,\dots, f_N$ satisfy \eqref{StandardAssumpf1thrufN}. We also note that the proof essentially follows from an observation made by Br\'ezis in Remark 2.3.7 of \cite{MR0348562}.  
\begin{thm}\label{mainthm2}
Suppose $f_1,\dots, f_N: V^N\rightarrow \R$ satisfy
\be\label{ThirdAssumpf1thrufN}
\begin{cases}
\partial_{x_j}f_j(x)=\{\nabla_{x_j}f_j(x)\}\;\text{for each $x\in V^N$,}
\\\\
\nabla_{x_j}f_j(x+ty)\rightharpoonup\nabla_{x_j}f_j(x)\;\text{as $t\rightarrow0^+$ for each $x,y\in V^N$, and}
\\\\
\displaystyle\sum^N_{j=1}(\nabla_{x_j}f_j(x)-\nabla_{x_j}f_j(y), x_j-y_j)\ge 0\;\; \text{for each $x,y\in V^N$},
\end{cases}
\ee
and define 
$$
{\cal D}=\{x\in V^N: \nabla_{x_1}f_1(x),\dots, \nabla_{x_N}f_N(x)\in H\}.
$$
Then for each $u^0\in {\cal D}$, there is a unique Lipschitz continuous $u:[0,\infty)\rightarrow H^N$ such that 
$u(t)\in {\cal D}$ for each $t\ge 0$ and 
\be\label{flowIVP3}
\begin{cases}
u(0)=u^0,\;\text{and}
\\\\
 \dot u_j(t)+ \nabla_{x_j}f_j(u(t))= 0\; \text{for a.e. $t\ge 0$ and each $j=1,\dots, N$}.
\end{cases}
\ee
\end{thm}
\begin{proof}
It suffices to show that $A: H^N\rightarrow 2^{H^N}$ defined via
\be\label{SpecialOperator}
Ax:=
\begin{cases}
\{(\nabla_{x_1}f_1(x),\dots,\nabla_{x_N}f_N(x))\},& x\in{\cal D}\\
\emptyset, & x\not\in{\cal D}
\end{cases}
\ee
($x\in H^N$) is maximal.  To this end, we first consider the operator $B: V^N\rightarrow (V^*)^N$ defined via
$$
Bx:=(x_1+\nabla_{x_1}f_1(x),\dots,x_N+\nabla_{x_N}f_N(x))
$$
for $x\in V^N$.  Notice that 
$$
B(x+ty)\rightharpoonup Bx
$$
and 
$$
(Bx-By,x-y)\ge \|x-y\|^2
$$
for all $x,y\in V^N$. 

\par By a theorem due independently to Browder \cite{MR156116} and Minty \cite{MR162159}, $B$ is surjective; see also Corollary 1.8 of \cite{MR1786735}.  In particular, if $y\in H^N$, there is $x\in V^N$ such that  
$$
Bx=y.
$$
That is, 
$$
x_j+\nabla_{x_j}f_j(x)=y_j
$$
for $j=1,\dots, N$.  Then $\nabla_{x_j}f_j(x)=y_j-x_j\in H$ for $j=1,\dots, N$, which implies $x\in {\cal D}$ and 
$$
x+Ax\ni y.
$$
\end{proof}
\begin{rem}\label{BrowderMintyRemark}
If $f_1,\dots, f_N$ additionally satisfy 
$$
\sum^N_{j=1}(\nabla_{x_j}f_j(x)-\nabla_{x_j}f_j(y), x_j-y_j)\ge \theta\|x-y\|^2
$$
($x,y\in V^N$), then $f_1,\dots, f_N$ has a unique Nash equilibrium and \eqref{ExpConvtoNash} holds for any 
$u(0)\in {\cal D}$.
\end{rem}

\section{Finite dimensional flows}\label{FiniteDimSect}
In this section, we will study a few implications of Theorem \ref{mainthm1} in the particular case 
$$
V=H=\R^d
$$
equipped with the standard Euclidean dot product and norm.  As there is just one topology to consider, the statements we'll make will be simpler than in the infinite dimensional setting. We will also identify a potential application to noncooperative games. 

\subsection{Compact domains}
Let $X=X_1\times\cdots\times X_N$, where $X_1,\dots, X_N\subset \R^d$ are convex, compact subsets with nonempty interior.  Further assume
\be
\begin{cases}
f_1,\dots, f_N: X\rightarrow \R \;\text{are continuous},
\\\\
X\ni y\mapsto \displaystyle\sum^N_{j=1}f_j(y_j,x_{-j})\;\text{is convex for each $\displaystyle x\in X$, and}
\\\\
\displaystyle\sum^N_{j=1}(\partial_{x_j}f_j(x)-\partial_{x_j}f_j(y))\cdot (x_j-y_j)\ge 0\;\; \text{for each $x,y\in X$.}
\end{cases}
\ee
As before we set 
$$
{\cal D}=\left\{x\in X: \bigcap^N_{j=1} (\partial_{x_j}f_j(x)+n_{X_j}(x))\neq \emptyset\right\},
$$
where
$$
n_{X_j}(z)=\{\zeta\in \R^d: \zeta\cdot(y-z)\le 0\;\text{for all $y\in X_j$}\}
$$
$(j=1,\dots, N$).  Theorem \ref{mainthm1} implies the following.   
\begin{prop}\label{CompactDomainFiniteD}
For each $u^0\in {\cal D}$, there  is a unique Lipschitz continuous $u: [0,\infty)\rightarrow \R^{Nd}$ with $u(t)\in {\cal D}$ for $t\ge 0$ and which satisfies 
\be
\begin{cases}
u(0)=u^0,\;\text{and}
\\\\
 \displaystyle\sum^N_{j=1}(\dot u_j(t)+ \partial_{x_j}f_j(u(t)))\cdot ( y_j-u_j(t))\ge 0\; \text{for a.e. $t\ge 0$ and each $y\in X$}.
\end{cases}
\ee
Furthermore, the limit
$$
\lim_{t\rightarrow\infty}\frac{1}{t}\int^t_0u(s)ds
$$
exists and equals a Nash equilibrium of $f_1,\dots, f_N$. 
\end{prop}
\begin{rem}
In view of Corollary \ref{ExistenceNashCoercive}, if 
$$
\displaystyle\sum^N_{j=1}(\partial_{x_j}f_j(x)-\partial_{x_j}f_j(y))\cdot( x_j-y_j)\ge \theta|x-y|^2
$$
for $x,y,\in X$, $f_1,\dots, f_N$ has a unique Nash equilibrium and the solution $u$ in the proposition 
above converges exponentially fast to this equilibrium point. 
\end{rem}

\subsection{An application to game theory}
Again suppose $X_1,\dots, X_N\subset \R^d$ are convex, compact subsets with nonempty interior.
Another interesting case to consider is when $f_1,\dots, f_N: X\rightarrow \R$ are each $N$-linear.  That is, 
\be\label{NlinearCond}
X_j\ni y_j\mapsto f_j(y_j,x_{-j})\;\text{is linear}
\ee
for each $j=1,\dots, N$ and $x\in X$.  These are precisely the types of 
functions Nash considered in his celebrated work on noncooperative games \cite{MR43432,MR31701}. 
Note that \eqref{NlinearCond} implies 
\be\label{NlinearCondDer}
y_j\cdot \nabla_{x_j}f_j(x)=f_j(y_j,x_{-j})
\ee
for $y_j\in X_j$, $x\in \R^{Nd}$, and $j=1,\dots, N$.  Consequently, 
\begin{align}
\sum^{N}_{j=1}\left(\nabla_{x_j}f_j(x)-\nabla_{x_j}f_j(y)\right)\cdot (x_j-y_j)=
\sum^{N}_{j=1}(f_j(x)-f_j(y_j,x_{-j})+f_j(y)-f_j(x_j,y_{-j}) )\ge 0
\end{align}
if and only if 
\be\label{CostCondition}
\sum^{N}_{j=1}(f_j(x)+f_j(y))\ge \sum^{N}_{j=1}(f_j(y_j,x_{-j})+f_j(x_j,y_{-j})).
\ee
With these assumptions, Theorem \ref{mainthm1} implies the following statement.

\begin{prop}
Suppose $f_1,\dots, f_N$ are $N$-linear and satisfy \eqref{CostCondition}. Then for each $u^0\in X$, 
there is a unique Lipschitz $u:[0,\infty)\rightarrow X$ such that 
\be\label{GameTheoryIVP}
\begin{cases}
u(0) = u^0\\\\
\displaystyle\sum^N_{j=1}(\dot u_j(t) +\nabla_{x_j}f_j(u(t)))\cdot (y_j- u_j(t))\ge 0,\;\text{for a.e. $t\ge 0$ and all $y\in X$.}
\end{cases}
\ee
Moreover, 
$$
\frac{1}{t}\int^t_0u(s)ds
$$
converges as $t\rightarrow\infty$ to a Nash equilibrium of $f_1,\dots, f_N$. 
\end{prop}
\begin{rem}
Using \eqref{NlinearCondDer}, we can rewrite the inequality in \eqref{GameTheoryIVP} as 
$$
\sum^N_{j=1}(\dot u_j(t)\cdot (y_j- u_j(t))+f_j(y_j,u_{-j}(t))-f_j(u(t)))\ge 0.
$$
\end{rem}

\subsection{The whole space}
Finally, we consider the case where $f_1,\dots,f_N$ are defined on the whole space $\R^{Nd}$. We will assume 
\be
\begin{cases}
f_1,\dots, f_N \text{ are continuous,}
\\\\ 
\displaystyle\R^{Nd}\ni y\mapsto \sum^N_{j=1}f_j(y_j,x_{-j}) \text{ is convex for each $x\in \R^{Nd}$, and}
\\\\
\displaystyle\sum^N_{j=1}(\partial_{x_j}f_j(x)-\partial_{x_j}f_j(y))\cdot (x_j-y_j) \ge 0 \text{ for each $x,y\in \R^{Nd}$}.
\end{cases}
\ee
Let us also set 
$$
{\cal D}:=\left\{x\in \R^{Nd}: \bigcap^N_{j=1}\partial_{x_j}f_j(x)\neq \emptyset \right\}.
$$
Theorem \ref{mainthm1} gives us the following proposition. 

\begin{prop}\label{FiniteDWholeSpaceProp}
For each $u^0\in{\cal D}$, there is a unique Lipschitz continuous $u: [0,\infty)\rightarrow \R^{Nd}$ such that $u(t)\in {\cal D}$ for each $t\ge 0$ and which satisfies 
\be
\begin{cases}
u(0) = u^0
\\\\
\;\dot u_j(t) +\partial_{x_j}f_j(u(t))\ni 0,\;\text{for a.e. $t\ge 0$ and $j=1,\dots, N$.}
\end{cases}
\ee
If $f_1,\dots, f_N$ has a
Nash equilibrium, then
$$
\lim_{t\rightarrow\infty}\frac{1}{t}\int^t_0u(s)ds
$$
exists and is also Nash equilibrium.
\end{prop}
\begin{rem}
If 
$$
\sum^N_{j=1}(\partial_{x_j}f_j(x)-\partial_{x_j}f_j(y))\cdot (x_j-y_j)\ge \theta|x-y|^2
$$
for all $x,y\in \R^{Nd}$ and some $\theta>0$, then $f_1,\dots, f_N$ 
has a unique Nash equilibrium at some $z\in \R^{Nd}$ and 
$$
|u(t)-z|\le e^{-\theta t}|u^0-z|,\quad t\ge 0.
$$
\end{rem}
\begin{rem}
A basic example we had in mind when considering Proposition \ref{FiniteDWholeSpaceProp} was  
$$
f_j(x)=\sum^N_{k,\ell=1}\frac{1}{2}A^j_{k,\ell}x_k\cdot x_\ell
$$
for $x=(x_1,\dots,x_N)\in \R^{Nd}$ and $j=1,\dots, N$. Here $A^j_{k,\ell}$ are $d\times d$ symmetric matrices which additionally satisfy 
$$
A^j_{k,\ell}=A^j_{\ell,k}
$$
$(j,k,\ell=1,\dots, N)$. The required monotonicity condition on $f_1,\dots, f_N$ is satisfied provided
\be\label{positivityAjjk}
\sum^{N}_{j,k=1}A^j_{k,j}y_j\cdot y_k\ge 0
\ee
for each $y\in \R^{Nd}$.
\end{rem}

\section{Examples in function spaces}\label{InFiniteDimSect}
In this final section, we will apply the ideas we have developed to consider functionals defined on Lebesgue and Sobolev spaces. As we have 
done above, we will consider both the existence of Nash equilibria and their approximation by continuous time flows. 
Throughout, we will assume $\Omega\subset \R^d$ is a bounded domain with smooth boundary. We will also change notation by using $v$ for the variables in which our functionals are defined. This allows us to reserve $x$ for points in $\Omega$. Finally, we will use $|\cdot |$ for any norm on a finite dimensional space. 

\subsection{A few remarks on Sobolev spaces}
Before discussing examples, let us recall a few facts about Sobolev spaces and establish some notation. A good reference for this material is Chapter 5 and 6 of \cite{MR2597943}. First we note  
$$
H^1_0(\Omega)\subset L^2(\Omega)\subset H^{-1}(\Omega)
$$
where each inclusion is compact.  Here $L^2(\Omega)$ is the space of Lebesgue measurable functions on $\Omega$ which are square integrable equipped with the standard inner product.
The space $H^1_0(\Omega)$ is the closure of all smooth functions having compact support in $\Omega$ in the norm 
$$
u\mapsto \left(\int_{\Omega}|Du|^2dx\right)^{1/2}.
$$

\par The topological dual of $H^1_0(\Omega)$ is denoted $H^{-1}(\Omega)$.  Recall that the Dirichlet Laplacian $-\Delta: H^1_0(\Omega)\rightarrow H^{-1}(\Omega)$ is defined via
$$
(-\Delta u,v)=\int_{\Omega}Du\cdot Dvdx
$$
for $u,v\in H^1_0(\Omega)$.  Here we are using $(\cdot,\cdot)$ as the natural pairing between $H^{-1}(\Omega)$ and $H^1_0(\Omega)$. Moreover, $-\Delta$ is an isometry; in particular, this map is invertible. This allows us to define the following inner product on $H^{-1}(\Omega)$
$$
(f,g)\mapsto(f,(-\Delta)^{-1}g).
$$
We also note that if $f\in L^2(\Omega)$, the inner product between $f$ and $g$ simplifies to 
\be\label{HminusInnerObs}
(f,(-\Delta)^{-1}g)=\int_\Omega f(-\Delta)^{-1}gdx.
\ee
\par In addition, we will consider the space $H^2(\Omega)\cap H^1_0(\Omega)$ consisting of $H^1_0(\Omega)$ functions for which $-\Delta u\in L^2(\Omega)$. This is a Hilbert space endowed with the inner product
$$
(u,v)\mapsto \int_\Omega \Delta u\;\Delta v dx.
$$

\subsection{$H=L^2(\Omega)$}
For $v\in H^1_0(\Omega)^N$, define
\be\label{FirstL2Examplefjay}
f_j(v)=\int_{\Omega}H_j(\nabla v_j)+F_j(v)-h_jv_j\;dx
\ee
for $j=1,\dots, N$. Here each $h_j\in L^2(\Omega)$, and $H_j: \R^d\rightarrow \R$ is assumed to be smooth and satisfy
\be\label{GjayUniformEst}
\theta |p-q|^2\le (\nabla H_j(p)-\nabla H_j(q))\cdot (p-q)\le \frac{1}{\theta}|p-q|^2
\ee
for all $p,q\in \R^d$ and some $\theta\in (0,1]$. We'll also suppose $F_1,\dots, F_N: \R^N\rightarrow \R$ are 
smooth, 
\be
\sum^N_{j=1}(\partial_{z_j}F_j(z)-\partial_{z_j}F_j(w))(z_j-w_j)\ge 0
\ee
for $z,w\in \R^N$, and there is $C$ such that 
\be\label{FjayAssump}
|F_j(z)|\le C(1+|z|^2)\;\text{and}\quad |\partial_{z_j}F_j(z)|\le C(1+|z|)
\ee
for $z\in \R^N$ and $j=1,\dots, N$.

\par Using these assumptions, it is straightforward to verify that $f_1,\dots, f_N$ fulfills \eqref{StandardAssumpf1thrufN} with $X_1=\cdots=X_N=V=H^1_0(\Omega).$    The following lemma also follows by routine computations.
\begin{lem}
For each $v,w\in H^1_0(\Omega)^N$ and $j=1,\dots, N$, 
$$
\partial_{v_j}f_j(v)=\{\nabla_{v_j}f_j(v)\}
$$
with 
$$
(\nabla_{v_j}f_j(v),w_j)=\int_\Omega \nabla H_j(\nabla v_j)\cdot \nabla w_j+\partial_{z_j}F_j(v)w_j-h_jw_jdx
$$
and
$$
\sum^N_{j=1}(\nabla_{v_j}f_j(v)-\nabla_{v_j}f_j(w),v_j-w_j)\ge \theta \int_{\Omega}|\nabla v-\nabla w|^2dx.
$$
Moreover, $v$ is a Nash equilibrium of $f_1,\dots, f_N$ if and only if it is a weak solution of 
\be
\begin{cases}
-\textup{div}(\nabla H_j(\nabla v_j))+\partial_{z_j}F_j(v)=h_j &\textup{in}\;\Omega\\
\hspace{1.87in}v_j=0 &\textup{on}\;\partial\Omega
\end{cases}
\ee
$(j=1,\dots, N)$.
\end{lem}
\begin{rem}
In this statement, $(\cdot,\cdot)$ denotes the natural pairing between $H^{-1}(\Omega)$ and $H^1_0(\Omega)$.  
\end{rem}

\par We now can state a result involving the approximation of Nash equilibrium for $f_1,\dots, f_N$. 
\begin{prop}
Let $u^0\in (H^2(\Omega)\cap H^1_0(\Omega))^N$. There is a unique Lipschitz continuous
$$
u: [0,\infty)\rightarrow L^2(\Omega)^N;  t\mapsto u(\cdot,t)
$$ 
such that $u(\cdot,t)\in (H^2(\Omega)\cap H^1_0(\Omega))^N$ for each $t\ge 0$ and 
\be\label{L2IVP}
\begin{cases}
\partial_t u_j =\textup{div}(\nabla H_j(\nabla u_j))-\partial_{z_j}F_j(u) +h_j\;\; &\textup{in}\;\Omega\times(0,\infty)\\
\;\;\;u_j=0 \;\; & \textup{on}\; \partial\Omega\times(0,\infty)\\
\;\;\;u_j=u^0_j \;\; & \textup{on}\; \Omega\times\{0\}
\end{cases}
\ee
$(j=1,\dots, N)$.  Furthermore, there is $\lambda>0$ such that 
\be
\int_{\Omega}|u(x,t)-v(x)|^2dx\le e^{-2\lambda t}\int_{\Omega}|u^0(x)-v(x)|^2dx
\ee
for $t\ge 0$. Here $v$ is the unique Nash equilibrium of $f_1,\dots,f_N$.
\end{prop}
\begin{rem}
For almost every $t\ge 0$, the PDE in \eqref{L2IVP} holds almost everywhere in $\Omega$; the boundary condition holds in the trace sense; and $u(\cdot, 0)=u^0$ as $L^2(\Omega)^N$ functions.
\end{rem}
\begin{proof}
Suppose that $\zeta\in L^2(\Omega)^N$ and $v\in H^1_0(\Omega)^N$ is a weak solution of 
\be
\begin{cases}
-\textup{div}(\nabla H_j(\nabla v_j))+\partial_{z_j}F_j(v)=h_j+\zeta_j &\textup{in}\;\Omega\\
\hspace{1.87in}v_j=0 &\textup{on}\;\partial\Omega.
\end{cases}
\ee
As $\partial\Omega$ is assumed to be smooth and in view of the hypothesis \eqref{GjayUniformEst}, elliptic regularity (Theorem 1, section 8.3 of \cite{MR2597943}) implies $v\in (H^2(\Omega)\cap H^1_0(\Omega))^N$. 
Therefore, 
\begin{align}
\left\{v\in H^1_0(\Omega)^N:\nabla_{v_1}f_1(v),\dots,\nabla_{v_N}f_N(v)\in L^2(\Omega) \right\}=(H^2(\Omega)\cap H^1_0(\Omega))^N.
\end{align}
Theorem \ref{mainthm1} and Corollary \ref{ExistenceNashCoercive} now allow us to conclude. 
\end{proof}

\subsection{Another example with $H=L^2(\Omega)$}
Define 
$$
f_j(v)=\int_\Omega L_j(\nabla v_1,\dots,\nabla v_N)-h_jv_j\;dx
$$
for $v\in H^1_0(\Omega)^N$.  Here $h_1,\dots, h_N\in L^2(\Omega)$ and each $L_j: \R^{Nd}\rightarrow \R$ is smooth with 
\be\label{LjUnifEst}
\displaystyle\theta|p-q|^2\le \sum^N_{j=1}(\nabla_{p_j}L_j(p)-\nabla_{p_j}L_j(q))\cdot(p_j-q_j)\le \frac{1}{\theta}|p-q|^2
\ee
for all $p,q\in\R^{Nd}$ and some $\theta\in (0,1]$. We'll also assume there is $C>0$ such that 
\be
|\nabla_{p_j}L_j(p)|\le C(1+|p|)
\ee 
for each $p\in \R^{Nd}$ and $j=1,\dots, N$.
 
\par Direct computation gives the following lemma. 
\begin{lem}
For each $v,w\in H^1_0(\Omega)^N$ and $j=1,\dots, N$, 
$$
\partial_{v_j}f_j(v)=\{\nabla_{v_j}f_j(v)\}
$$
with 
\be
(\nabla_{v_j}f_j(v),w_j)=\int_\Omega\nabla_{p_j}L_j(\nabla v_1,\dots,\nabla v_N)\cdot \nabla w_j-h_jw_j\;dx
\ee
and
$$
\sum^N_{j=1}(\nabla_{v_j}f_j(v)-\nabla_{v_j}f_j(w),v_j-w_j)\ge \theta \int_{\Omega}|\nabla v-\nabla w|^2dx.
$$
Moreover, 
$$
\nabla_{v_j}f_j(v+tw)\rightharpoonup \nabla_{v_j}f_j(v)
$$
as $t\rightarrow 0^+$, and $v$ is a Nash equilibrium of $f_1,\dots, f_N$ if and only if it is a weak solution of 
\be
\begin{cases}
-\textup{div}(\nabla_{p_j} L_j(\nabla v_1,\dots, \nabla v_N))=h_j &\textup{in}\;\Omega\\
\hspace{1.88in}v_j=0 &\textup{on}\;\partial\Omega
\end{cases}
\ee
$(j=1,\dots, N)$.
\end{lem}

\par With this lemma, we can now establish the subsequent assertion about the Nash equilibrium of $f_1,\dots, f_N$. 
\begin{prop}
Let $u^0\in (H^2(\Omega)\cap H^1_0(\Omega))^N$. There is a unique Lipschitz continuous 
$$
u: [0,\infty)\rightarrow L^2(\Omega)^N;  t\mapsto u(\cdot,t)
$$ 
such that $u(\cdot,t)\in (H^2(\Omega)\cap H^1_0(\Omega))^N$ for each $t\ge 0$ and  
\be
\begin{cases}
\partial_t u_j =\textup{div}(\nabla_{p_j} L_j(\nabla u_1,\dots, \nabla u_N)) +h_j\;\; &\textup{in}\;\Omega\times(0,\infty)\\
\;\;\;u_j=0 \;\; & \textup{on}\; \partial\Omega\times(0,\infty)\\
\;\;\;u_j=u^0_j \;\; & \textup{on}\; \Omega\times\{0\}
\end{cases}
\ee
$(j=1,\dots, N$). Furthermore, there is $\lambda>0$ such that 
\be
\int_{\Omega}|u(x,t)-v(x)|^2dx\le e^{-2\lambda t}\int_{\Omega}|u^0(x)-v(x)|^2dx
\ee
for $t\ge 0$. Here $v$ is the unique Nash equilibrium of $f_1,\dots,f_N$.
\end{prop}
\begin{proof}
Suppose that $\xi\in L^2(\Omega)^N$ and $v\in H^1_0(\Omega)^N$ is a weak solution of 
\be
\begin{cases}
-\textup{div}(\nabla_{p_j} L_j(\nabla v_1,\dots, \nabla v_N))=h_j+\xi_j &\textup{in}\;\Omega\\
\hspace{1.88in}v_j=0 &\textup{on}\;\partial\Omega.
\end{cases}
\ee
By the uniform estimate \eqref{LjUnifEst}, elliptic regularity implies $v\in (H^2(\Omega)\cap H^1_0(\Omega))^N$ (see the remark at the end of section 9.1 in \cite{MR2597943}).  As a result, 
\begin{align}
\left\{v\in H^1_0(\Omega)^N:\nabla_{v_1}f_1(v),\dots,\nabla_{v_N}f_N(v)\in L^2(\Omega) \right\}=(H^2(\Omega)\cap H^1_0(\Omega))^N.
\end{align}
We then conclude by Theorem \ref{mainthm2} and Remark \ref{BrowderMintyRemark}. 
\end{proof}

\subsection{$H=H^{-1}(\Omega)$}
For a given $v\in L^2(\Omega)^N$, set 
$$
f_j(v)=\int_{\Omega}F_j(v) dx-\langle h_j,v_j\rangle
$$
for $j=1,\dots, N$. Here $h_1,\dots, h_N\in H^{-1}(\Omega)$, and $\langle\cdot,\cdot\rangle$ denotes the  inner product in $H^{-1}(\Omega)$.  
We will assume $F_1,\dots, F_N: \R^N\rightarrow\R$ are smooth and satisfy \eqref{FjayAssump} and 
\be\label{FjayAssump2}
\begin{cases}
|\partial_{z_k}\partial_{z_j}F_j(z)|\le C\\\\
\displaystyle\theta|z-w|^2\le \sum^N_{i=1}(\partial_{z_i}F_i(z)-\partial_{z_i}F_i(w))(z_i-w_i)
\end{cases}
\ee
for $z,w\in \R^N$, $j,k=1,\dots, N$, and some $C,\theta>0$. For convenience, we will also assume 
\be\label{FjayAssump3}
\partial_{z_j}F_j(0)=0
\ee
for $j=1,\dots, N$.

\par The following observations can be verified by routine computations.  
\begin{lem}
For each $v,w\in L^2(\Omega)^N$ and $j=1,\dots, N$, 
$$
\partial_{v_j}f_j(v)=\{\nabla_{v_j}f_j(v)\}
$$
with 
\be
\nabla_{v_j}f_j(v)=\partial_{z_j}F_j(v)-(-\Delta)^{-1}h_j.
\ee
Moreover, $v\in H^1_0(\Omega)^N$ is a Nash equilibrium of $f_1,\dots, f_N$ if and only if it is a weak solution of 
\be\label{HminusNashCond}
\begin{cases}
-\Delta(\partial_{z_j}F_j(v))=h_j &\textup{in}\;\Omega\\
\hspace{.84in}v_j=0 &\textup{on}\;\partial\Omega
\end{cases}
\ee
$(j=1,\dots, N)$.
\end{lem}

\par By Proposition \ref{DualityProp} in the appendix, there are smooth functions $G_1,\dots, G_N$ such that 
\be\label{GjayCond}
\begin{cases}
\partial_{z_j}F_j(z)=y_j\;\text{if and only if}\;\partial_{y_j}G_j(y)=z_j \\\\
|G_j(y)|\le B(1+|y|^2)\;\text{and}\quad |\partial_{y_j}G_j(y)|\le B(1+|y|)\\\\
\displaystyle 0\le \sum^{N}_{i=1}\left(\partial_{y_i}G_i(y)-\partial_{y_i}G_i(z)\right) (y_i-z_i)\\\\
\displaystyle |\partial_{y_k}\partial_{y_j}G_j(y)|\le \frac{1}{\theta}
\end{cases}
\ee
for all $y,z\in \R^N$, $j,k=1,\dots, N$ and some constant $B$.  Notice that in view of \eqref{FjayAssump3}
\be
\partial_{y_j}G_j(0)=0
\ee
for $j=1,\dots, N$. Thus, it is possible to solve \eqref{HminusNashCond} by choosing 
$$
w_j=(-\Delta)^{-1}h_j\in H^1_0(\Omega)
$$
for $j=1,\dots, N$ and then setting 
$$
v_j=\partial_{y_j}G_j(w)
$$
for $j=1,\dots, N$. In particular, we have just shown that $f_1,\dots, f_N$ has a Nash equilibrium, and it is not hard to 
see it is unique. 

\par Let us now see how to approximate this Nash equilibrium. 
\begin{prop}
Let $u^0\in H^1_0(\Omega)^N$. There is a unique Lipschitz continuous
$$
u: [0,\infty)\rightarrow H^{-1}(\Omega)^N;  t\mapsto u(\cdot,t)
$$ 
such that $u(\cdot,t)\in H^1_0(\Omega)^N$ for each $t\ge 0$ and  
\be
\begin{cases}
\partial_t u_j =\Delta(\partial_{z_j}F_j(u)) -h_j\;\; &\textup{in}\;\Omega\times(0,\infty)\\
\;\;\;u_j=0 \;\; & \textup{on}\; \partial\Omega\times(0,\infty)\\
\;\;\;u_j=u^0_j \;\; & \textup{on}\; \Omega\times\{0\}
\end{cases}
\ee
$(j=1,\dots, N$). Furthermore, there is $\lambda>0$ such that 
\be
\|u(\cdot,t)-v\|\le e^{-\lambda t}\| u^0-v\|
\ee
for $t\ge 0$. Here $\|\cdot \|$ denotes the norm on $H^{-1}(\Omega)^N$ and $v$ is the Nash equilibrium of $f_1,\dots,f_N$.
\end{prop}
\begin{proof}
Define $A: H^{-1}(\Omega)^N\rightarrow 2^{H^{-1}(\Omega)^N}$ as
$$
Av
=\begin{cases}
\{(-\Delta(\partial_{z_1}F_1(v)) +h_1,\dots, -\Delta(\partial_{z_N}F_N(v)) +h_N)\}, & v\in H^1_0(\Omega)^N\\
\emptyset, &\text{otherwise}
\end{cases}
$$
for $v\in H^{-1}(\Omega)^N$.  We note that $D(A)=H^1_0(\Omega)^N$, and $0\in Av$ if and only if $v\in H^1_0(\Omega)^N$ is a Nash equilibrium for $f_1,\dots, f_N$.
By the second inequality listed in \eqref{FjayAssump2},
\begin{align}
\langle Av-Aw,v-w\rangle&= \int_\Omega\sum^N_{j=1}(\partial_{z_j}F_j(v)-\partial_{z_j}F_j(w))(v_j-w_j)dx\\
&\ge \theta \int_{\Omega}|v-w|^2dx\\
&\ge \lambda \|v-w\|^2
\end{align}
for $v,w\in H^1_0(\Omega)^N$ and some $\lambda>0$.

\par As a result, it suffices to check that $A$ is maximal. This amounts to finding a weak solution $v\in H^1_0(\Omega)^N$ of 
\be\label{AhMinusoneMaximal}
\begin{cases}
v_j-\Delta(\partial_{z_j}F_j(v))=h_j +\zeta_j &\textup{in}\;\Omega\\
\hspace{1.05in}v_j=0 &\textup{on}\;\partial\Omega
\end{cases}
\ee
for a given $\zeta\in H^{-1}(\Omega)^N$.   To this end, we consider the auxiliary problem of finding $w\in H^1_0(\Omega)^N$ such that
\be\label{AhMinusoneMaximal2}
\begin{cases}
\partial_{y_j}G_j(w)-\Delta w_j=h_j +\zeta_j &\textup{in}\;\Omega\\
\hspace{.98in}w_j=0 &\textup{on}\;\partial\Omega.
\end{cases}
\ee
If we can find $w$, then $v\in H^1_0(\Omega)^N$ defined by 
$$
v_j=\partial_{y_j}G_j(w)
$$
$(j=1,\dots, N)$ is a solution of \eqref{AhMinusoneMaximal}. 

\par Notice that \eqref{AhMinusoneMaximal2} has a solution if and only if 
\be\label{geejayAuxFuns}
g_j(w)=\int_{\Omega}G_j(w)+\frac{1}{2}|\nabla w_j|^2dx- (h_j +\zeta_j,w_j) 
\ee
($w\in H^1_0(\Omega)^N$, $j=1,\dots, N$) has a Nash equilibrium.  Here $(\cdot,\cdot)$ is the natural pairing between $H^{-1}(\Omega)$ and $H^1_0(\Omega)$. 
Using \eqref{GjayCond}, it is straightforward to check that $g_1,\dots, g_N$ satisfies the properties listed in \eqref{StandardAssumpf1thrufN} with $X_1=\dots =X_N=V=H^1_0(\Omega)$. 
As a result,  Corollary \ref{ExistenceNashCoercive} implies that $g_1,\dots, g_N$ has a Nash equilibrium.
\end{proof}

\subsection{$H=H^{1}_0(\Omega)$}
Set 
$$
f_j(v)=\int_{\Omega}F_j(\Delta v) -\nabla h_j\cdot \nabla v_j \;dx
$$
for $v\in (H^2(\Omega)\cap H^1_0(\Omega))^N$ and $j=1,\dots, N$. Here $h_1,\dots, h_N\in H^1_0(\Omega)$ and 
$$
\Delta v=(\Delta v_1,\dots, \Delta v_N).
$$
We will suppose $F_1,\dots, F_N: \R^N\rightarrow \R$ are smooth 
and satisfy \eqref{FjayAssump2} and \eqref{FjayAssump3}.  A few key observations regarding $f_1,\dots f_N$ are as follows. 

\begin{lem}
For each $v,w\in (H^2(\Omega)\cap H^1_0(\Omega))^N$ and $j=1,\dots, N$, 
$$
\partial_{v_j}f_j(v)=\{\nabla_{v_j}f_j(v)\}
$$
with 
$$
(\nabla_{v_j}f_j(v),w_j)=\int_\Omega \partial_{z_j}F_j(\Delta v)\Delta w_j-\nabla h_j\cdot \nabla w_jdx.
$$
Moreover, $v$ is a Nash equilibrium of $f_1,\dots, f_N$ if and only if it is a solution of 
\be
\begin{cases}\label{NashFjayHplusOne}
-\partial_{z_j}F_j(\Delta v)=h_j &\textup{in}\;\Omega\\
\hspace{0.71in}v_j=0 &\textup{on}\;\partial\Omega
\end{cases}
\ee
$(j=1,\dots, N)$.
\end{lem}
\begin{rem}
Above, $(\cdot,\cdot)$ denotes the natural pairing between $(H^2(\Omega)\cap H^1_0(\Omega))^*$ and $H^2(\Omega)\cap H^1_0(\Omega)$.  
\end{rem}

\par Using the functions $G_1,\dots, G_N$ from the previous subsection, we can solve 
\be
\begin{cases}
-\Delta v_j=-\partial_{y_j}G_j(-h) &\textup{in}\;\Omega\\
\hspace{0.26in}v_j=0 &\textup{on}\;\partial\Omega,
\end{cases}
\ee
for $j=1,\dots, N$ to obtain a solution of \eqref{NashFjayHplusOne}. It follows that $f_1,\dots, f_N$ has a unique Nash equilibrium $v$, which belongs to the space  
$$
{\cal D}=\left\{v\in(H^2(\Omega)\cap H^1_0(\Omega))^N: \Delta v\in H^1_0(\Omega)^N\right\}.
$$
As with our previous examples, we can approximate this Nash equilibrium with a continuous time flow. 

\begin{prop}
Let $u^0\in {\cal D}$. There is a unique Lipschitz continuous 
$$
u: [0,\infty)\rightarrow H^{1}_0(\Omega)^N;  t\mapsto u(\cdot,t)
$$ 
such that $u(\cdot,t)\in {\cal D}$ for each $t\ge 0$ and  
\be
\begin{cases}
\partial_t u_j =\partial_{z_j}F_j(\Delta u) +h_j\;\; &\textup{in}\;\Omega\times(0,\infty)\\
\;\;\;u_j=0 \;\; & \textup{on}\; \partial\Omega\times(0,\infty)\\
\;\;\;u_j=u^0_j \;\; & \textup{on}\; \Omega\times\{0\}
\end{cases}
\ee
$(j=1,\dots, N$). Furthermore, there is $\lambda>0$ such that 
\be
\int_{\Omega}|\nabla u(x,t)-\nabla v(x)|^2dx\le e^{-2\lambda t}\int_{\Omega}| \nabla u^0(x)- \nabla v(x)|^2dx
\ee
for $t\ge 0$. Here $v$ is the unique Nash equilibrium of $f_1,\dots,f_N$.
\end{prop}
\begin{proof}
Define $A: H^1_0(\Omega)^N\rightarrow 2^{H^1_0(\Omega)^N}$ by 
\be
Av=
\begin{cases}
\{ \left(-\partial_{z_1}F_1(\Delta v) -h_1,\dots,-\partial_{z_N}F_N(\Delta v) -h_N\right) \}, & v\in {\cal D}\\
\emptyset, & \text{otherwise}
\end{cases}
\ee
for $v\in H^1_0(\Omega)^N$.  Note that $D(A)={\cal D}$, and $0\in Av$ if and only if $v\in {\cal D}$ is a Nash equilibrium for $f_1,\dots, f_N$. In view of \eqref{FjayAssump2},
\begin{align}
\langle Av-Aw,v-w\rangle &=\int_\Omega\sum^N_{j=1}(\partial_{z_j}F_j(\Delta v)-\partial_{z_j}F_j(\Delta w) )(\Delta v_j-\Delta w_j)dx\\
& \ge \theta \int_{\Omega}|\Delta v-\Delta w|^2dx\\
&\ge \lambda  \int_{\Omega}|\nabla v-\nabla w|^2dx.
\end{align}
for $v,w\in {\cal D}$ and some $\lambda>0$.  

\par Therefore, it suffices to show 
\be
\begin{cases}\label{MaximalFjayHplusOne}
v_j-\partial_{z_j}F_j(\Delta v)=h_j+\zeta_j &\textup{in}\;\Omega\\
\hspace{.93in}v_j=0 &\textup{on}\;\partial\Omega
\end{cases}
\ee
has a solution $v\in (H^2(\Omega)\cap H^1_0(\Omega))^N$ for a given $\zeta\in H^1_0(\Omega)^N$. We note that any such solution $v$ will automatically 
belong to ${\cal D}$. Moreover, this is equivalent to finding a weak solution $v\in H^1_0(\Omega)^N$ of
\be
\begin{cases}\label{MaximalFjayHplusOne2}
\partial_{y_j}G_j(v-h-\zeta)-\Delta v_j= 0&\textup{in}\;\Omega\\
\hspace{1.55in}v_j=0 &\textup{on}\;\partial\Omega.
\end{cases}
\ee
We emphasize that elliptic regularity implies that any such weak solution actually belongs to $(H^2(\Omega)\cap H^1_0(\Omega))^N$.

\par Any $v\in H^1_0(\Omega)^N$ which solves \eqref{MaximalFjayHplusOne2} is a Nash equilibrium of $g_1,\dots, g_N$ defined by 
\be
g_j(v)=\int_{\Omega}G_j(v-h-\zeta)+\frac{1}{2}|\nabla v_j|^2 dx.
\ee
As we did when considering the functions defined in \eqref{geejayAuxFuns}, we can apply Corollary \ref{ExistenceNashCoercive} and  conclude that $g_1,\dots, g_N$ has a unique Nash 
equilibrium. 
\end{proof}

\appendix
\section{Dual functions}
In this appendix, we will verify a technical assertion that was used to establish the existence of a Nash equilibrium in a few of the examples we considered above.
\begin{prop}\label{DualityProp}
Suppose $F_1,\dots, F_N: \R^N\rightarrow \R$ are smooth and satisfy 
\be\label{FjayAssumpLast}
|F_j(z)|\le C(1+|z|^2)
\ee
for $z\in \R^N$ and $j=1,\dots, N$ and 
\be\label{JointConvexity}
 \theta |z-w|^2\le \sum^{N}_{j=1}\left(\partial_{z_j}F_j(z)-\partial_{z_j}F_j(w)\right) (z_j-w_j)
\ee
for $z,w\in \R^N$ and some $\theta>0$.  There are smooth functions $G_1,\dots,G_N: \R^N\rightarrow \R$ with the following properties. 
\begin{enumerate}[(i)]

\item $\partial_{z_j}F_j(z)=y_j$ if and only if $\partial_{y_j}G_j(y)=z_j$.

\item For all $j,k=1,\dots, N$ and $y\in \R^N$,
\be
|\partial_{y_k}\partial_{y_j}G_j(y)|\le \frac{1}{\theta}.
\ee

\item For each $y,w\in \R^N$,
\be\label{JointConvexityGjay}
0\le \sum^{N}_{j=1}\left(\partial_{y_j}G_j(y)-\partial_{y_j}G_j(w)\right) (y_j-w_j).
\ee

\item There is a constant $B$ such that 
\be\label{GjayAssump}
|G_j(y)|\le B(1+|y|^2)\;\text{and}\quad |\partial_{y_j}G_j(y)|\le B(1+|y|)
\ee
for $y\in \R^N$ and $j=1,\dots, N$.

\end{enumerate}
\end{prop}

\begin{proof}
First we note that the mapping of $\R^N$
\be\label{GradientMapRN}
z\mapsto (\partial_{z_1}F_1(z),\dots,\partial_{z_N}F_N(z))
\ee
is invertible. This follows from the theorem due to Browder \cite{MR156116} and Minty \cite{MR162159} as \eqref{JointConvexity} implies that this map is both monotone and coercive;  invertibility is a consequence of Hadamard's global invertibility theorem (Theorem A of \cite{MR305418}), as well. The condition \eqref{JointConvexity} also gives \be\label{HessianInequalityLeg}
 \theta|v|^2\le\sum^N_{j,k=1}\partial_{z_k}\partial_{z_j}F_j(z)v_kv_j
\ee
for each $z,v\in \R^N$.

 \par Now fix $y\in \R^N$ and let $z\in \R^N$ be the unique Nash equilibrium of 
\be\label{feyeshiftedFun}
\R^N\ni x\mapsto F_j(x)-y_jx_j
\ee
$(j=1,\dots, N)$.  Such a Nash equilibrium exists for these $N$ functions by Corollary \ref{ExistenceNashCoercive}.  As a result,
\be\label{DerivativeDualCond}
\partial_{z_j}F_j(z)=y_j, \quad j=1,\dots, N.
\ee
Since \eqref{GradientMapRN} is invertible and smooth, \eqref{DerivativeDualCond} determines $z$ as a smooth
function of $y$. Let us now define 
\be\label{DefinitionGjay}
G_j(y):=z_jy_j-F_j(z)
\ee
for $j=1,\dots, N$. Differentiating with respect to $y_j$ and using \eqref{DerivativeDualCond}, we find 
$$
\partial_{y_j}G_j(y)=z_j.
$$
This proves $(i)$ and that
\be\label{GradientMapRN2}
y\mapsto (\partial_{y_1}G_1(y),\dots,\partial_{y_N}G_N(y))
\ee
is the inverse mapping of \eqref{GradientMapRN}.

\par It follows from part $(i)$ and the inverse function theorem that whenever $\partial_{y_j}G_j(y)=z_j$ and 
$\partial_{z_j}F_j(z)=y_j$ for $j=1,\dots, N$, the matrices 
$$
(\partial_{z_k}\partial_{z_j}F_j(z))^N_{j,k=1}\text{ and  }(\partial_{y_k}\partial_{y_j}G_j(y))^N_{j,k=1}
$$ 
are inverses. Consequently, for a given $w\in \R^N$, we can choose $v\in \R^N$ defined as
$$
v_j=\sum^N_{k=1}\partial_{y_k}\partial_{y_j}G_j(y)w_k
$$
in \eqref{HessianInequalityLeg} to get 
\begin{align}\label{LowerMonEstG}
0&\le \theta\sum^N_{j=1}\left(\sum^N_{k=1}\partial_{y_k}\partial_{y_j}G_j(y)w_k\right)^2\nonumber\\ 
&\le\sum^N_{j,k=1}\partial_{y_k}\partial_{y_j}G_j(y)w_kw_j\\
&=\sum^N_{j=1}\left(\sum^N_{k=1}\partial_{y_k}\partial_{y_j}G_j(y)w_k\right)w_j\\
&\le \sqrt{\sum^N_{j=1}\left(\sum^N_{k=1}\partial_{y_k}\partial_{y_j}G_j(y)w_k\right)^2}\cdot |w|.
\end{align}

\par As a result,
\be
\sum^N_{j=1}\left(\sum^N_{k=1}\partial_{y_k}\partial_{y_j}G_j(y)w_k\right)^2\le\frac{1}{\theta^2}|w|^2.
\ee
Upon selecting $w=e_k$, the $k$th standard unit vector in $\R^N$, we find
$$
(\partial_{y_k}\partial_{y_j}G_j(y))^2=\left(\sum^N_{\ell=1}\partial_{y_\ell}\partial_{y_j}G_j(y)w_\ell\right)^2\le \sum^N_{i=1}\left(\sum^N_{\ell=1}\partial_{y_\ell}\partial_{y_i}G_i(y)w_\ell\right)^2
\le \frac{1}{\theta^2}
$$
for $j,k=1,\dots, N$.  We conclude assertion $(ii)$.  

\par Part $(iii)$ follows directly from \eqref{LowerMonEstG}. Indeed, for any $y,w\in \R^N$, we can use the fundamental 
theorem of calculus to derive
\begin{align}
&\sum^{N}_{j=1}\left(\partial_{y_j}G_j(y)-\partial_{y_j}G_j(w)\right)(y_j-w_j)\\
&=\int^1_0\left(\sum^{N}_{j,k=1}\partial_{y_k}\partial_{y_j}G_j(w+t(y-w))(y_k-w_k)(y_j-w_j)\right)dt\\
&\ge 0.
\end{align}

\par We are left to verify assertion $(iv)$. First note that by part $(ii)$, 
\begin{align}
|\partial_{y_j}G_j(y)| & \le|\partial_{y_j}G_j(y)-\partial_{y_j}G_j(0)|+|\partial_{y_j}G_j(0)|\\
&=\left|\int^1_0\sum^{N}_{k=1}\partial_{y_k}\partial_{y_j}G_j(ty)y_k dt\right|+|\partial_{y_j}G_j(0)|\\
&\le\sum^{N}_{k=1}\left(\int^1_0\left|\partial_{y_k}\partial_{y_j}G_j(ty)\right|dt\right) |y_k| +|\partial_{y_j}G_j(0)|\\
&\le \frac{1}{\theta}\sum^{N}_{k=1}|y_k|+|\partial_{y_j}G_j(0)|\\
&\le D(1+|y|)
\end{align}
for some $D>0$ independent of $j=1,\dots, N$ and $y$.

\par For a given $y\in \R^N$ there is a unique $z$ such that $\partial_{z_j}F_j(z)=y_j$ and $\partial_{y_j}G_j(y)=z_j$ for $j=1,\dots, N$. We just showed above that 
$$
|z_j|=|\partial_{y_j}G_j(y)|\le D(1+|y|).
$$
In view of hypothesis, \eqref{FjayAssumpLast} and our definition of $G_j$ \eqref{DefinitionGjay}, we additionally have 
\begin{align}
|G_j(y)|&=|z_jy_j-F_j(z)|\\
&\le |z_j||y_j|+|F_j(z)|\\
&\le  D(1+|y|)|y|+C(1+|z|^2)\\
&\le  D(1+|y|)|y|+C(1+ND^2(1+|y|)^2)\\
&\le B(1+|y|^2)
\end{align}
for an appropriately chosen constant $B$.  
\end{proof}

\bibliography{NEbib}{}

\bibliographystyle{plain}

\typeout{get arXiv to do 4 passes: Label(s) may have changed. Rerun}

\typeout{get arXiv to do 4 passes: Label(s) may have changed. Rerun}

\end{document}